\documentclass[graybox]{svmult}

\usepackage{type1cm}        
\usepackage{makeidx}         
\usepackage{graphicx}        
\usepackage{multicol}        
\usepackage[bottom]{footmisc}

\usepackage{newtxtext}       %
\usepackage[varvw]{newtxmath}       
\usepackage{bussproofs}
\usepackage{proof}

\makeindex             
                       
\newcommand{\LK}{{\bf LK}}
\newcommand{\LE}{{\bf L}\varepsilon}
\newcommand{\LKp}{{\bf LK}^{+}}
\newcommand{\LKpp}{{\bf LK}^{++}}                       

\begin{document}

\title*{Epsilon Calculus Provides Shorter Cut-Free Proofs}
\author{Matthias Baaz\orcidID{0000-0002-7815-2501} and\\ Anela Loli{\'c}\orcidID{0000-0002-4753-7302}}
\institute{Matthias Baaz \at Institute of Discrete Mathematics and Geometry, TU Wien, \email{baaz@logic.at}
\and Anela Loli{\'c} \at Kurt G\"odel Society, Institute of Logic and Computation, TU Wien \email{anela@logic.at}}
%
%
\maketitle

\abstract{In this paper we show that cut-free derivations in the epsilon format of sequent calculus provide for a non-elementary speed-up w.r.t. cut-free proofs in usual sequent calculi in first-order language.}

\section{Introduction}
Epsilon calculus gives the impression to provide shorter proofs than other proof mechanisms. To make this claim precise, we compare in this paper an epsilon calculus variant of $\LK$ with $\LK$ and related calculi. The main property of epsilon calculus used is its ability to overbind bound variables.

\section{Epsilon Calculus}
The $\varepsilon$-calculus uses $\varepsilon$-terms to represent $\exists x A(x)$ by $A(\varepsilon_x A(x))$. Consequently, $\forall x A(x)$ is represented by $A(\varepsilon_x \neg A(x))$. 
As the $\varepsilon$-calculus is only based on the representation by critical formulas $$A(t) \to A(\varepsilon_x A(x))$$
for $A(t) \to \exists x A(x)$ and propositional axioms and rules, the unrestricted deduction theorem of propositional calculus transfers to this formalization of first-order logic: 
The $\varepsilon$-proof itself is a tautology 
$$(\bigwedge_{i=1}^{n} A_i (t_i ) \to A_i (\varepsilon_x A_i (x))) \to E,$$
where $E$ is the original result translated into $\varepsilon$-calculus. 
Note that strong quantifier inferences are replaced by substitutions of $\varepsilon_x \neg A(x)$ for $\forall x A(x)$ positive and $\varepsilon_x A(x)$ for $\exists x A(x)$ negative. (Valid propositional formulas do not influence an $\varepsilon$-proof.)
The extended first $\varepsilon$-theorem \cite{hilbertbernays39,moser2006epsilon} eliminates algorithmically the critical formulas obtaining a Herbrand disjunction $\bigvee_{i=1}^m E(\overline{t_i})$, where $E$ is the $\varepsilon$-translation of $\exists \overline{x} E'(\overline{x})$, $E'$ being quantifier-free.
The argument can be easily extended to formulas $E'$ which contain only weak quantifiers.

The language of epsilon calculus is based on the term language of epsilon expressions and other function symbols and on propositional language otherwise.

\section{$\LE$, $\LK$, and Related Sequent Calculi} \label{sec.3}
To compare cut-free derivations we consider a sequent calculus format of the epsilon calculus.
\begin{definition}[$\LE$]
(In the language of epsilon calculus) \\[1ex]
Axiom schema: $A \vdash A$, $A$ atomic.\\[1ex]
The inference rules are:\\[1ex]
\begin{itemize}
	 \item for conjunction

	\medskip

	\begin{minipage}{0.4\linewidth}
	\centering
	\begin{prooftree}
	\AxiomC{$A, B, \Gamma \vdash \Delta$}
	\RightLabel{$\land_{l}$}
	\UnaryInfC{$A \land B, \Gamma \vdash \Delta$}
	\end{prooftree}
	\end{minipage}
	\begin{minipage}{0.5\linewidth}
	\centering
	\begin{prooftree}
	\AxiomC{$\Gamma_1 \vdash \Delta_1, A$}
	\AxiomC{$\Gamma_2 \vdash \Delta_2, B$}
	\RightLabel{$\land_{r}$}
	\BinaryInfC{$\Gamma_1 , \Gamma_2 \vdash \Delta_1 , \Delta_2 , A \land B$}
	\end{prooftree}
	\end{minipage}
	
	 \item for disjunction

	\medskip
	
	\begin{minipage}{0.5\linewidth}
	\centering
	\begin{prooftree}
	\AxiomC{$A, \Gamma_1 \vdash \Delta_1$}
	\AxiomC{$B, \Gamma_2 \vdash \Delta_2$}
	\RightLabel{$\lor_{l}$}
	\BinaryInfC{$A \lor B, \Gamma_1, \Gamma_2 \vdash \Delta_1 , \Delta_2$}
	\end{prooftree}
	\end{minipage}
	\begin{minipage}{0.4\linewidth}
	\centering
	\begin{prooftree}
	\AxiomC{$\Gamma \vdash \Delta, A, B$}
	\RightLabel{$\lor_{r}$}
	\UnaryInfC{$\Gamma \vdash \Delta, A \lor B$}
	\end{prooftree}
	\end{minipage}
	
	\item for implication

	\medskip

	\begin{minipage}{0.5\linewidth}
	\centering
	\begin{prooftree}
	\AxiomC{$\Gamma_1 \vdash \Delta_1, A$}
	\AxiomC{$B, \Gamma_2 \vdash \Delta_2$}
	\RightLabel{$\to_{l}$}
	\BinaryInfC{$A \to B, \Gamma_1 , \Gamma_2 \vdash \Delta_1 , \Delta_2$}
	\end{prooftree}
	\end{minipage}
	\begin{minipage}{0.4\linewidth}
	\centering
	\begin{prooftree}
	\AxiomC{$A, \Gamma \vdash \Delta, B$}
	\RightLabel{$\to_{r}$}
	\UnaryInfC{$\Gamma \vdash \Delta, A \to B$}
	\end{prooftree}
	\end{minipage}
	
	\item for negation

	\medskip

    \begin{minipage}{0.4\linewidth}
	\centering
	\begin{prooftree}
	\AxiomC{$\Gamma \vdash \Delta, A$}
	\RightLabel{$\neg_{l}$}
	\UnaryInfC{$\neg A, \Gamma \vdash \Delta$}
	\end{prooftree}
    \end{minipage}
	\begin{minipage}{0.4\linewidth}
	\centering
	\begin{prooftree}
	\AxiomC{$A, \Gamma \vdash \Delta$}
	\RightLabel{$\neg_{r}$}
	\UnaryInfC{$\Gamma \vdash \Delta, \neg A$}
	\end{prooftree}
	\end{minipage}
	
	\item weakening

	\begin{minipage}{0.4\linewidth}
	\centering
	\begin{prooftree}
	\AxiomC{$\Gamma \vdash \Delta$}
	\RightLabel{$w_{l}$}
	\UnaryInfC{$A, \Gamma \vdash \Delta$}
	\end{prooftree}
	\end{minipage}
	\begin{minipage}{0.4\linewidth}
	\centering
	\begin{prooftree}
	\AxiomC{$\Gamma \vdash \Delta$}
	\RightLabel{$w_{r}$}
	\UnaryInfC{$\Gamma \vdash \Delta, A$}
	\end{prooftree}
	\end{minipage}
	
	\item contraction

	\begin{minipage}{0.4\linewidth}
	\centering
	\begin{prooftree}
	\AxiomC{$A, A, \Gamma \vdash \Delta$}
	\RightLabel{$c_{l}$}
	\UnaryInfC{$A, \Gamma \vdash \Delta$}
	\end{prooftree}
	\end{minipage}
	\begin{minipage}{0.4\linewidth}
	\centering
	\begin{prooftree}
	\AxiomC{$\Gamma \vdash \Delta, A, A$}
	\RightLabel{$c_{r}$}
	\UnaryInfC{$\Gamma \vdash \Delta, A$}
	\end{prooftree}
	\end{minipage}

	\item cut 

\begin{minipage}{0.6\linewidth}
	\centering
	\begin{prooftree}
	\AxiomC{$\Gamma_1 \vdash \Delta_1, A$}
	\AxiomC{$A, \Gamma_2 \vdash \Delta_2$}
	\RightLabel{$cut$}
	\BinaryInfC{$\Gamma_1 , \Gamma_2 \vdash \Delta_1 , \Delta_2$}
\end{prooftree}
\end{minipage}

	\item quantifier inferences:
	
	the weak quantifier inferences $\exists_r$
\begin{prooftree}
	\AxiomC{$\Pi \vdash \Delta, A(t)$}
	\RightLabel{$\exists_r$}
	\UnaryInfC{$\Pi \vdash \Delta, A(\varepsilon_x A(x))$}
\end{prooftree}

	and $\forall_l$
\begin{prooftree}
	\AxiomC{$A(t), \Pi \vdash \Delta$}
	\RightLabel{$\forall_l$}
	\UnaryInfC{$A(\varepsilon_x \neg A(x)), \Pi \vdash \Delta$}
\end{prooftree}
the strong quantifier inferences $\exists_l$: replaced by substitution\\[1ex]
and $\forall_r$: replaced by substitution.
\end{itemize}
\end{definition}
We have to define first a translation of an expression in first-order language to an expression in epsilon calculus language.
\begin{definition}
Let $A$ be a formula. Its epsilon translation is denoted as $[A]^\varepsilon$ and inductively defined as
\begin{itemize}
	\item $A$ is an atom. Then $[A]^\varepsilon = A$.
	\item $A = B \circ C$, where $\circ \in \{\land, \lor, \to\}$ and $B$ and $C$ formulas. Then $[A]^\varepsilon = [B]^\varepsilon \circ [C]^\varepsilon$.
	\item $A = \exists x A'(x)$. Then $[A]^\varepsilon = [A'(\varepsilon_x A'(x))]^\varepsilon$.
	\item $A = \forall x A'(x)$. Then $[A]^\varepsilon = [A'(\varepsilon_x \neg A'(x))]^\varepsilon$.
\end{itemize}
\end{definition}
$[A]^{\forall \exists}$ is a translation form epsilon calculus language to first-order language when $A = [B]^\varepsilon$ for some expression $B$, and undefined otherwise. 
\begin{example}
Note that $[A]^{\forall \exists}$ for an epsilon calculus expression $A$ does not always exist: let $A$ be $\varepsilon_v (v = \varepsilon_x \neg x = x) = \varepsilon_x \neg x = x$.
\end{example}
\begin{proposition} \label{prop.1}
Every $\LK$-derivation possibly with cuts can be translated into an $\LE$-derivation of equal or smaller length.
\end{proposition}
\begin{proof}
All inference steps are replaced by corresponding inference steps with exception of strong quantifier rules, which are replaced by substitution.
\end{proof}
\begin{remark}
Note that the usual form of epsilon proofs can be obtained by deleting the quantifier inferences of $\LE$, and replacing them by 
	\begin{prooftree}
		\AxiomC{$(\psi')$}
		\noLine
		\UnaryInfC{$\Pi' \vdash \Delta', A'(t)$}
		\AxiomC{$A'(\varepsilon_x A(x)) \vdash A'(\varepsilon_x A(x))$}
		\RightLabel{$\exists_r$}
		\BinaryInfC{$A'(t) \to A'(\varepsilon_x A(x)), \Pi' \vdash \Delta'$}
	\end{prooftree}
and
	\begin{prooftree}
		\AxiomC{$(\psi')$}
		\noLine
		\UnaryInfC{$A'(t), \Pi' \vdash \Delta'$}
		\AxiomC{$A'(\varepsilon_x \neg A(x)) \vdash A'(\varepsilon_x \neg A(x))$}
		\RightLabel{$\forall_l$}
		\BinaryInfC{$A'(\varepsilon_x \neg A(x)) \to A'(t), \Pi' \vdash \Delta'$}
		\doubleLine
		\UnaryInfC{$\neg A'(t) \to \neg A'(\varepsilon_x \neg A(x)), \Pi' \vdash \Delta'$}
	\end{prooftree}

\end{remark}
Recall that a function on the natural numbers is elementary if it can be defined by a quantifier-free formula from $+$, $\times$, and the function $x \to 2^x$. By independent results of R. Statman \cite{statman1979lower} and of V. P. Orevkov \cite{orevkov1982lower}, the sizes of the smallest cut-free {\LK}-proofs of sequents of length $n$ are not bounded by any elementary function on $n$.
\begin{example}
A shortest cut-free $\LK$-derivation of $\exists y (A(y) \to \forall x A(x))$ is

\begin{prooftree}
	\AxiomC{$A(a) \vdash A(a)$}
	\RightLabel{$w_r$}
	\UnaryInfC{$A(a) \vdash A(a), \forall x A(x)$}
	\RightLabel{$\to_r$}
	\UnaryInfC{$\vdash A(a), A(a) \to \forall x A(x)$}
	\RightLabel{$\exists_r$}
	\UnaryInfC{$\vdash A(a), \exists y (A(y) \to \forall x A(x))$}
	\RightLabel{$\forall_r$}
	\UnaryInfC{$\vdash \forall x A(x), \exists y (A(y) \to \forall x A(x))$}
	\RightLabel{$w_l$}
	\UnaryInfC{$A(b) \vdash \forall x A(x), \exists y (A(y) \to \forall x A(x))$}
	\RightLabel{$\to_r$}
	\UnaryInfC{$\vdash A(b) \to \forall x A(x), \exists y (A(y) \to \forall x A(x))$}
	\RightLabel{$\exists_r$}
	\UnaryInfC{$\vdash \exists y (A(y) \to \forall x A(x)), \exists y (A(y) \to \forall x A(x))$}
	\RightLabel{$c_r$}
	\UnaryInfC{$\vdash \exists y (A(y) \to \forall x A(x))$}
\end{prooftree}
Its translation to $\LE$ is 

\begin{prooftree}
	\AxiomC{$A(e) \vdash A(e)$}
	\RightLabel{$w_r$}
	\UnaryInfC{$A(e) \vdash A(e), A(e)$}
	\RightLabel{$\to_r$}
	\UnaryInfC{$\vdash A(e), A(e) \to A(e)$}
	\RightLabel{$\exists_r$}
	\UnaryInfC{$\vdash A(e), A(f) \to A(e)$}
	\RightLabel{$(*) \ + w_l$}
	\UnaryInfC{$A(b) \vdash A(e), A(f) \to A(e)$}
	\RightLabel{$\to_r$}
	\UnaryInfC{$\vdash A(b) \to A(e), A(f) \to A(e)$}
	\RightLabel{$\exists_r$}
	\UnaryInfC{$\vdash A(f) \to A(e), A(f) \to A(e)$}
	\RightLabel{$c_r$}
	\UnaryInfC{$\vdash A(f) \to A(e) \ (= \exists y (A(y) \to \forall x A(x))^\varepsilon)$}
\end{prooftree}
where $e \equiv \varepsilon_x \neg A(x)$ and $f \equiv \varepsilon_y (A(y) \to A(\varepsilon_x \neg A(x)))$.\\ [1ex]
$(*)$: $\forall_r$ has been replaced by the substitution of $\varepsilon_x \neg A(x)$ for $a$.\\ [1ex]
The shortest cut-free derivation of $\vdash A(f) \to A(e)$ in $\LE$ is however

\begin{prooftree}
	\AxiomC{$A(e) \vdash A(e)$}
	\RightLabel{$\to_l$}
	\UnaryInfC{$\vdash A(e) \to A(e)$}
	\RightLabel{$\exists_r$}
	\UnaryInfC{$\vdash A(f) \to A(e)$}
\end{prooftree}
\end{example}
\begin{theorem}[\cite{orevkov1982lower,statman1979lower}] \label{th.orevkov}
There is a specific family of sequents $\{S_i\}_{i<\omega}$ described in \cite{DBLP:journals/fuin/BaazL94} and due to Statman \cite{statman1979lower}, and specific {\LK}-proofs thereof, such that they have the following properties:
\begin{enumerate}
	\item the size of $S_i$ is polynomial in $i$;
	\item there is no bound on the size of their smallest cut-free {\LK}-proofs that is elementary in $i$;
	\item the size of these proofs (with cuts), however, is polynomially bounded in $i$.
\end{enumerate}
\end{theorem}
In the following we will consider the sequence of sequents $\{S_i\}_{i<\omega}$ from Theorem \ref{th.orevkov} above.
\begin{corollary}
Each worst-case sequence as formulated in Theorem \ref{th.orevkov} generates a worst-case sequence, where the end-sequents contain weak quantifiers only.
\end{corollary}
\begin{proof}
Strong quantifiers in a cut-free $\LK$ proof can be replaced by Skolem functions without lengthening the proof or introducing cuts \cite{baaz2011methods}.
\end{proof}
\begin{definition}
The matrix $A^M$ of a first-order formula $A$ is $A$, after deletion of all quantifiers and after replacement of bound variables by free variables.
\end{definition}
\begin{example}
$[\exists x (\forall y A(x,y) \lor B(x))]^M = A(a,b) \lor B(a)$.
\end{example}
\begin{lemma}
There is a specific family of sequents $\{S_i\}_{i<\omega}$ such that they have the following properties:
\begin{enumerate}
	\item the size of $S_i$ is polynomial in $i$;
	\item there is no bound on the size of their smallest cut-free {\LK}-proofs that is elementary in $i$;
	\item the size of these proofs (with cuts), however, is polynomially bounded in $i$;
	\item they contain only weak quantifiers;
	\item on the left-side of the conclusion for every cut $A$, $\forall \overline{x} (A^M \to A^M)$ is added.
\end{enumerate}
\end{lemma}
\begin{proof}
For the proofs with cut the addition of $\forall \overline{x} (A^M \to A^M)$ might lead to even shorter proofs, for the cut-free proofs the proofs may be double exponentially shorter if the newly added universal formulas are eliminated in the following way: 
In the moment where the corresponding implication left is inferred, replace this inference by a cut. In consequence, there is a proof with propositional cuts only, which can be eliminated in at most double exponential expense \cite{DBLP:journals/tcs/Weller11}.
\end{proof}
\begin{theorem}
There is a sequence of cut-free $\LE$-proofs such that 
\begin{enumerate}
	\item the size of $S_i$ is polynomial in $i$;
	\item the end-sequents $S_i$ are translations of first-order sequents $S'_i$ with weak quantifiers only;
	\item the size of these proofs, however, is polynomially bounded in $i$;
	\item there is no bound on the size of the smallest cut-free {\LK}-proofs of the translation of $S_i$ to first-order language that is elementary in $i$.
\end{enumerate}
\end{theorem}
\begin{proof}
We choose a sequence of $\LK$-proofs from the lemma above.
We translate the proofs with cut into epsilon calculus (this does not lengthen the proof according to Proposition \ref{prop.1}).
In the $\LE$-proof we replace all cuts on $A$ by inferences of $A \to A$ on the left side.
Derive immediately $[\forall \overline{x} (A^M \to A^M)]^\epsilon$.
Contract it with $\forall \overline{x} (A^M \to A^M)$ which is already in the end-sequent.
\end{proof}
Note that the extended first epsilon theorem \cite{hilbertbernays39} provides an upper bound for cut-free $\LE$-derivations in the rough size of $2^{\left. 2^{.^{.^{.}}} \right\}i}$ for the $i$-th cut-free $\LE$-derivation.
The question remains however, whether $\LE$-derivations with cuts can be translated into $\LK$-derivations with cuts in an elementary way.

\section{$\LKp$ and $\LKpp$}
Another example of the speed-up of cut-free proofs as in Section \ref{sec.3} relates to the sequent calculi $\LKp$ and $\LKpp$ introduced in \cite{DBLP:journals/jsyml/AguileraB19}. 
They are obtained from $\LK$ by weakening the eigenvariable conditions. The resulting calculi are therefore globally but possibly not locally sound. This means that all derived statements are true but that not every sub-derivation is meaningful. 

Note that there is already a non-elementary speed-up of cut-free proofs of $\LKp$, or $\LKpp$ w.r.t. cut-free {\LK}-proofs \cite{DBLP:journals/jsyml/AguileraB19}.
In contrast, the transformation of cut-free $\LKpp$-proofs into cut-free $\LKp$-proofs is elementary bounded \cite{DBLP:conf/wollic/BaazL23}.
\begin{definition}[side variable relation $<_{\varphi, \LK}$, cf. \cite{DBLP:journals/jsyml/AguileraB19}]
Let $\varphi$ be an $\LK$-derivation. We say $b$ is a side variable of $a$ in $\varphi$ (written $a <_{\varphi, \LK} b$) if $\varphi$ contains a strong quantifier inference of the form

\begin{prooftree}
	\AxiomC{$\Gamma \vdash \Delta, A(a,b, \overline{c})$}
	\RightLabel{$\forall_r$}
	\UnaryInfC{$\Gamma \vdash \Delta, \forall x A(x,b,\overline{c})$}
\end{prooftree}
or of the form

\begin{prooftree}
	\AxiomC{$A(a,b, \overline{c}), \Gamma \vdash \Delta$}
	\RightLabel{$\exists_l$}
	\UnaryInfC{$\exists x A(x,b,\overline{c}), \Gamma \vdash \Delta$}
\end{prooftree}
We may omit the subscript ${\varphi, \LK}$ in $<_{\varphi, \LK}$ if it is clear from the context.
\end{definition}
In addition to strong and weak quantifier inferences we define $\LKp$-suitable quantifier inferences.
\begin{definition}[$\LKp$-suitable quantifier inferences, cf. \cite{DBLP:journals/jsyml/AguileraB19}] \label{def.5}
We say a quantifier inference is suitable for a proof $\varphi$ if either it is a weak quantifier inference, or the following three conditions are satisfied:
\begin{itemize}
	\item (substitutability) the eigenvariable does not appear in the conclusion of $\varphi$.
	\item (side variable condition) the relation $<_{\varphi, \LK}$ is acyclic.
	\item (weak regularity) the eigenvariable of an inference is not the eigenvariable of another strong quantifier inference in $\varphi$.
\end{itemize}
\end{definition}
\begin{definition}[$\LKp$, cf. \cite{DBLP:journals/jsyml/AguileraB19}] 
$\LKp$ is obtained from $\LK$ by replacing the usual eigenvariable conditions by $\LKp$-suitable ones.
\end{definition}
Similarly to $\LKp$, we define the calculus $\LKpp$ by further weakening the eigenvariable conditions 
\begin{definition}[$\LKpp$-suitable quantifier inferences, cf. \cite{DBLP:journals/jsyml/AguileraB19}] \label{def.7}
We say a quantifier inference is suitable for a proof $\varphi$ if either it is a weak quantifier inference, or it satisfies  
\begin{itemize}
	\item substitutability,
	\item the side variable condition, and 
	\item (very weak regularity) the eigenvariable of an inference with main formula $A$ is different to the eigenvariable of an inference with main formula $A'$ whenever $A \not = A'$.
\end{itemize}
\end{definition}

\begin{definition}[$\LKpp$, cf. \cite{DBLP:journals/jsyml/AguileraB19}]
$\LKpp$ is obtained from $\LK$ by replacing the usual eigenvariable conditions by $\LKpp$-suitable ones.
\end{definition}
\begin{theorem} \label{TheoremCorrectness}
\mbox{ }
\begin{enumerate}
\item If a sequent is $\LKp$-derivable, then it is already $\LK$-derivable.
\item If a sequent is $\LKpp$-derivable, then it is already $\LK$-derivable.
\end{enumerate}
\end{theorem}
\begin{proof}[Proof Sketch]
Consider an $\LKpp$-proof $\varphi$ (an $\LKp$ proof is also an $\LKpp$-proof). Replace every universal quantifier inference unsound w.r.t. $\LK$ by an $\to_l$ inference:
\begin{prooftree}
	\AxiomC{$\Gamma \vdash \Delta, A(a)$}
	\AxiomC{$\forall x A(x) \vdash \forall x A(x)$}
	\RightLabel{$\to_l$}
	\BinaryInfC{$\Gamma, A(a) \to \forall x A(x) \vdash \Delta, \forall x A(x)$}
\end{prooftree}
Similarly, replace every existential quantifier inference unsound w.r.t. $\LK$ by an $\to_l$ inference:
\begin{prooftree}
	\AxiomC{$\exists x A(x) \vdash \exists x A(x)$}
	\AxiomC{$A(a), \Gamma \vdash \Delta$}
	\RightLabel{$\to_l$}
	\BinaryInfC{$\Gamma, \exists x A(x), \exists x A(x) \to A(a) \vdash \Delta$}
\end{prooftree}
By doing this, we obtain a proof of the desired sequent, together with formulas of the form 
$$A(a) \to \forall x A(x) \quad \mbox{or} \quad \exists x A(x) \to A(a)$$
on the left-hand side. Note that the resulting derivation does not contain any inference based on eigenvariable conditions. We can eliminate each of the additional formulas on the left-hand side by adding an existential quantifier inference and cutting with sequents of the form
$$\vdash \exists y (A(y) \to \forall x A(x))$$
or of the form
$$\vdash \exists y (\exists x A(x) \to A(y)),$$
both of which are easily derivable. For more details see \cite{DBLP:journals/jsyml/AguileraB19}.
\end{proof}
\begin{example} Consider the following locally unsound but globally sound $\LKp$-derivation $\varphi$:
\begin{prooftree}
	\AxiomC{$A(a) \vdash A(a)$}
	\RightLabel{$\forall_r$}
	\UnaryInfC{$A(a) \vdash \forall y A(y)$}
	\RightLabel{$\to_r$}
	\UnaryInfC{$\vdash A(a) \to \forall y A(y)$}
	\RightLabel{$\exists_r$}
	\UnaryInfC{$\vdash \exists x (A(x) \to \forall y A(y))$}
\end{prooftree}
As $a$ is the only eigenvariable the side variable relation $<_{\varphi, \LK}$ is empty.
\end{example} 
The focus in \cite{DBLP:journals/jsyml/AguileraB19} has been on the strongly reduced complexity of cut-free $\LKp$- and $\LKpp$-proofs (Theorem $2.6$ and Corollary $2.7$).

Note that all three conditions of Definition \ref{def.5} and Definition \ref{def.7} are necessary.
\begin{example}
If substitutability is violated, the following derivation is possible
\begin{prooftree}
	\AxiomC{$A(a) \vdash A(a)$}
	\RightLabel{$\forall_r$}
	\UnaryInfC{$A(a) \vdash \forall x A(x)$}
\end{prooftree}
If the side variable relation is not acyclic, the following derivation $\varphi$ is possible (with the side variable conditions $a <_{\varphi, \LK} b$ and $b <_{\varphi, \LK} a$, which loop)

\begin{prooftree}
	\AxiomC{$A(a,b) \vdash A(a,b)$}
	\RightLabel{$\forall_r$}
	\UnaryInfC{$A(a,b) \vdash \forall y A(a,y)$}
	\RightLabel{$\exists_r$}
	\UnaryInfC{$A(a,b) \vdash \exists x \forall y A(x,y)$}
	\RightLabel{$\exists_l$}
	\UnaryInfC{$\exists x A(x,b) \vdash \exists x \forall y A(x,y)$}
	\RightLabel{$\forall_l$}
	\UnaryInfC{$\forall y \exists x A(x,y) \vdash \exists x \forall y A(x,y)$}
\end{prooftree}
If weak regularity is violated, the following derivation is possible
\begin{prooftree}
	\AxiomC{$A(a) \vdash A(a)$}
	\RightLabel{$\forall_r$}
	\UnaryInfC{$A(a) \vdash \forall x A(x)$}
	\RightLabel{$\exists_l$}
	\UnaryInfC{$\exists y A(y) \vdash \forall x A(x)$}
\end{prooftree}
\end{example}
\begin{lemma} \label{lem.2}
There is a specific family of sequents $\{S_i\}_{i<\omega}$ with the following properties:
\begin{enumerate}
	\item the size of $S_i$ is polynomial in $i$;
	\item there is no bound on the size of their smallest cut-free $\LKp$-proofs (or $\LKpp$-proofs) that is elementary in $i$;
	\item the size of these proofs (with cuts), however, is polynomially bounded in $i$;
	\item the end-sequents have only weak quantifiers;
	\item on the left-side of the conclusion for every cut $A$, $\forall \overline{x} (A^M \to A^M)$ is added.
\end{enumerate}
\end{lemma}
\begin{proof}
Note that Skolemization is not possible by direct substitution into strong quantifiers. However, Skolemization can be performed by adding additional cuts, which lengthen the proof linearly:

\begin{prooftree}
	\AxiomC{$\Pi \vdash \Gamma, A(a,t)$}
	\UnaryInfC{$\Pi \vdash \Gamma, \forall x A(x,t)$}
	\AxiomC{$A(f(t),t) \vdash A(f(t),t)$}
	\UnaryInfC{$\forall x A(x,t) \vdash A(f(t),t)$}
	\BinaryInfC{$\Pi \vdash \Gamma, A(f(t),t)$}
\end{prooftree}
and
\begin{prooftree}
	\AxiomC{$A(f(t),t) \vdash A(f(t),t)$}
	\UnaryInfC{$A(f(t),t) \vdash \exists x A(x,t)$}
	\AxiomC{$A(a,t), \Pi \vdash \Gamma$}
	\UnaryInfC{$\exists x A(x,t), \Pi \vdash \Gamma$}
	\BinaryInfC{$A(f(t),t), \Pi \vdash \Gamma$}
\end{prooftree}
\end{proof}
\begin{theorem}
There is a sequence of cut-free $\LE$-proofs such that 
\begin{enumerate}
	\item the size of $S_i$ is polynomial in $i$;
	\item the end-sequents $S_i$ are translations of first-order sequents $S'_i$ with weak quantifiers only;
	\item the size of these proofs, however, is polynomially bounded in $i$;
	\item there is no bound on the size of the smallest cut-free $\LKp$- or $\LKpp$-proofs of the translation of $S_i$ to first-order language that is elementary in $i$.
\end{enumerate}
\end{theorem}
\begin{proof}
We choose a sequence of $\LKp$- or $\LKpp$- proofs according to Lemma \ref{lem.2}.
We translate the proofs with cut into epsilon calculus (this does not lengthen the proof according to Proposition \ref{prop.1}).
In the $\LE$-proof we replace all cuts on $A$ by inferences of $A \to A$ on the left side.
Derive immediately $[\forall \overline{x} (A^M \to A^M)]^\epsilon$.
Contract it with the $\forall \overline{x} (A^M \to A^M)$ which is already in the end-sequent.
Note that cut-free $\LKp$- or $\LKpp$-proofs with end-sequents with weak quantifiers only are $\LK$-proofs.
\end{proof}

\section{Conclusion}
The effect that arbitrary cuts in $\LE$ can be transferred into universal cuts with linear increase of length demonstrates that no cut-elimination for $\LE$ by induction on the size of cut-formulas is possible. 
This implies that e.g. Gentzen-style cut-elimination and Sch\"utte-Tait-style cut-elimination are not feasible. Here the fundamental different nature of the (extended) first epsilon theorem becomes obvious \cite{hilbertbernays39}.

%

\bibliographystyle{plain}
\bibliography{refs}

\begin{thebibliography}{1}

\bibitem{DBLP:journals/jsyml/AguileraB19}
Juan~P. Aguilera and Matthias Baaz.
\newblock Unsound inferences make proofs shorter.
\newblock {\em J. Symb. Log.}, 84(1):102--122, 2019.

\bibitem{DBLP:journals/fuin/BaazL94}
Matthias Baaz and Alexander Leitsch.
\newblock On skolemization and proof complexity.
\newblock {\em Fundam. Informaticae}, 20(4):353--379, 1994.

\bibitem{baaz2011methods}
Matthias Baaz and Alexander Leitsch.
\newblock {\em Methods of Cut-elimination}, volume~34.
\newblock Springer Science \& Business Media, 2011.

\bibitem{DBLP:conf/wollic/BaazL23}
Matthias Baaz and Anela Lolic.
\newblock Effective skolemization.
\newblock In {\em WoLLIC}, volume 13923 of {\em Lecture Notes in Computer
  Science}, pages 69--82. Springer, 2023.

\bibitem{hilbertbernays39}
David Hilbert and Paul Bernays.
\newblock Grundlagen der {M}athematik.
\newblock 2, 1939.

\bibitem{moser2006epsilon}
Georg Moser and Richard Zach.
\newblock The epsilon calculus and {H}erbrand complexity.
\newblock {\em Studia Logica}, 82(1):133--155, 2006.

\bibitem{orevkov1982lower}
Vladimir~P Orevkov.
\newblock Lower bounds for increasing complexity of derivations after cut
  elimination.
\newblock {\em Journal of Soviet Mathematics}, 20:2337--2350, 1982.

\bibitem{statman1979lower}
Richard Statman.
\newblock Lower bounds on herbrand’s theorem.
\newblock {\em Proceedings of the American Mathematical Society},
  75(1):104--107, 1979.

\bibitem{DBLP:journals/tcs/Weller11}
Daniel Weller.
\newblock On the elimination of quantifier-free cuts.
\newblock {\em Theor. Comput. Sci.}, 412(49):6843--6854, 2011.

\end{thebibliography}
\end{document}